\documentclass[11pt]{amsart}

\usepackage{mathrsfs}

\usepackage{amsmath,latexsym,amssymb,amsthm,array,amsfonts}

\usepackage{mathrsfs,dsfont,bbm}
\theoremstyle{plain}
\newtheorem{theorem}{Theorem}[section]
\newtheorem{lemma}[theorem]{Lemma}
\newtheorem{corollary}[theorem]{Corollary}
\newtheorem{proposition}[theorem]{Proposition}

\numberwithin{equation}{section}

\begin{document}

\title[On the linear fractional self-attracting diffusion]
{On the linear fractional self-attracting diffusion${}^{*}$}

\footnote[0]{${}^{*}$The Project-sponsored by NSFC (10571025) and
the Key Project of Chinese Ministry of Education (No.106076).}

\date{}

\author[L. Yan, Y. Sun and Y. Lu]{Litan Yan, Yu Sun and Yunsheng Lu}



\maketitle

\begin{center}
{\it Department of Mathematics, Donghua University,\\ 2999 North
Renmin Rd. Shanghai 201620, P. R. China.}
\end{center}

\maketitle


\begin{abstract}
In this paper, we introduce the linear fractional self-attracting
diffusion driven by a fractional Brownian motion with Hurst index
$1/2<H<1$, which is analogous to the linear self-attracting
diffusion. For $1$-dimensional process we study its convergence and
the corresponding weighted local time. For $2$-dimensional process,
as a related problem, we show that the renormalized
self-intersection local time exists in $L^2$ if
$\frac12<H<\frac3{4}$.
\end{abstract}

\section{Introduction}
In 1991, Durrett and Rogers~\cite{Durr} studied a system that models
the shape of a growing polymer. Under some conditions, they
established asymptotic behavior of the solution of stochastic
differential equation
\begin{equation}\label{Int1.1}
X_t={B_t}+\int_0^t\int_0^s\Phi(X_s-X_u)duds,
\end{equation}
where $B$ is a $d$-dimensional standard Brownian motion and $\Phi$
Lipschitz continuous. If $\Phi(x)=\Psi(x)x/\|x\|$ and $\Psi(x)\geq
0$, $X_t$ is a continuous analogue of a process introduced by
Diaconis and studied by Pemantle~\cite{Peman1}. The path dependent
stochastic differential equation can be considered as polymer model.
In 1995, Cranston and Le Jan~\cite{Crans} extended the model and
introduced self-attracting diffusions, where for $d=1$ two cases of
are studied: the linear interaction where $\Phi$ is a linear
function and the constant interaction in dimension $1$, where
$\Phi(x)=\sigma{\rm {sign}}(x)$ for positive $\sigma$, and in both
cases the almost sure convergence of $X_t$ is proved.
Herrmanna-Roynette~\cite{Herr1}, Herrmanna-Scheutzowb~\cite{Herr2}
generalized these results.

On the other hand, the statistical properties of fractional Brownian
motion (fBm) are used to construct a path integral representation of
the conformations of some polymers (see, for examples, Chakravarti
and Sebastian~\cite{Chak}, Cherayil and Biswas~\cite{Chere},
Sebastian~\cite{Seba}). Thus, as a natural extension to
\eqref{Int1.1} one may consider the path dependent stochastic
differential equation of the form
\begin{equation}\label{Int1.6}
X^H_t={B_t^H}+\int_0^t\int_0^s\Phi(X_s^H-X_u^H)duds,
\end{equation}
where $B^H$ is a $d$-dimensional fractional Brownian motion with
Hurst index $H\in (0,1)$ and $\Phi$ Lipschitz continuous. Then it is
not difficult to show that the above equation admits a unique strong
solution. We will call the solution {\it the fractional
self-attracting diffusion} driven by fBm. In this paper, we consider
only a particular case as follows ({\it the linear fractional
self-attracting diffusion}):
\begin{equation}\label{1.7}
X^H_t={B_t^H}-a\int_0^t\int_0^s(X_s^H-X_u^H)duds+\nu t
\end{equation}
with $a>0$, $\nu\in {\mathbb R}^d$ and $\frac12<H<1$. Our aims are
to study the convergence and local times of the processes given
by~\eqref{1.7} with $d=1$. As a related problem, for the two
dimensional process we shall show that the renormalized
self-intersection local time exists in $L^2$ if
$\frac12<H<\frac3{4}$.

The structure of this paper is as follows. In Section~\ref{sec2} we
briefly recall fBm and related the It\^o type stochastic integral.
In Section~\ref{sec3} we investigate convergence of the linear
fractional self-attracting diffusion. We show that the process
converges with probability one as $t$ tends to infinity. In
Section~\ref{sec4}, we define the weighted local time of the process
and obtain a Meyer-Tanaka type formula. Finally, in
Section~\ref{sec5} for $2$-dimensional process we show that its
renormalized self-intersection local time exists in $L^2$ if
$\frac12<H<\frac3{4}$.

\section{Fractional Brownian motion and the It\^o type formula}
\label{sec2}
In this section, we briefly recall the definition and properties of
stochastic integral with respect to fBm. Throughout this paper we
assume that $\frac12<H<1$ is arbitrary but fixed. Let
$(\Omega,{\mathcal F},\mu)$ be a complete probability space such
that a fractional Brownian motion with Hurst index $H$ is well
defined. For simplicity we let $C$ stand for a positive constant
depending only on the subscripts and its value may be different in
different appearance, and this assumption is also adaptable to $c$.

Recall that a centered continuous Gaussian process
$B^H=\{B_t^H,\,t\geq 0\}$ with the covariance function
$$
E\left[B_t^HB_s^H\right]=\frac12\left(t^{2H}+s^{2H}-|t-s|^{2H}
\right), \;s,t\geq 0,
$$
is called {\it the fBm with Hurst index} $H$. Here $E$ denotes the
expectation with respect to the probability law of $B^H$ on
$\Omega$. This process was first introduced by Kolmogorov and
studied by Mandelbrot and Van Ness~\cite{Man1}, where a stochastic
integral representation in terms of a standard Brownian motion was
established. The definition of stochastic integrals with respect to
the fBm has been investigated by several authors. Here, we refer to
Duncan {\it et al}~\cite{Dun} and Hu-{\O}ksendal~\cite{Hu4} (see
also Elliott-Van der Hoek~\cite{Elli}, Hu~\cite{Hu3},
Nualart~\cite{Nua3,Nua4}, {\O}ksendel~\cite{Oks}) for the definition
and the properties of the fractional It\^o integral
$$
\int_0^tu_sdB^H_s
$$
of an adapted process $u$. For $1/2<H<1$ we define the function
$\phi:{\mathbb R}_+\times{\mathbb R}_+\to{\mathbb R}_+$ by
$$
\phi(s,t)=H(2H-1)|s-t|^{2H-2},\quad s,t\geq 0.
$$
Recall that the Malliavin $\phi$-derivative of the function
$U:\,\Omega\to {\mathbb R}$ defined in~\cite{Dun} as follows:
$$
D_s^\phi U=\int_0^\infty \phi(r,s)D_rUdr,
$$
where $D_rU$ is the fractional Malliavin derivative at $r$. Define
the space ${\mathbb L}_\phi^{1,2}$ to be the set of measurable
processes $u$ such that $D_t^\phi u_s$ exists for a.a. $s,t\geq 0$
and
\begin{equation}\label{condition1}
 \|u\|^2_{{\mathbb
L}^{1,2}_\phi}:=E\left[\int_0^\infty\int_0^\infty D_s^\phi
u_tD_t^\phi u_sdsdt+\int_0^\infty\!\!\int_0^\infty
u_{s}u_{t}\phi(s,t) dsdt\right]\!<\infty.
\end{equation}
Thus, the integral $\int_0^\infty u_sdB_s^H$ can be well defined as
an element of $L^2(\mu)$ if $u$ satisfies~\eqref{condition1}. For
the integral process $\eta_t=\int_0^tu_sdB_s^H,$ we have (see, for
examples, Duncan {\it et al}~\cite{Dun}, Hu~\cite{Hu3})
$$
D_s^\phi\eta_t=\int_0^tD_s^\phi u_rdB_r^H+\int_0^tu_r\phi(s,r)dr.
$$
In particular, if $u$ is deterministic, then
$D_s^\phi\eta_t=\int_0^tu_r\phi(s,r)dr.$

\begin{theorem}[Duncan {\it et al}~\cite{Dun}, Hu~\cite{Hu3}]\label{theorem2.1}
Let $F\in C^{2}({\mathbb R})$ having polynomial growth and let the
process $X$ be given as follows:
$$
dX_t=v_tdt+u_tdB_t^H,\quad X_0=x\in {\mathbb R},
$$
where $u\in {\mathbb L}_\phi^{1,2}$ and measurable process $v$
satisfies $\int_0^t|v_s|ds<\infty$ $a.s$. Then we have, for all
$t\geq 0$
\begin{align}\label{Ito2-1}
F(X_t)=F(x)+\int_0^t\frac{\partial }{\partial x}F(s,X_s)dX_s
 +\int_0^t\frac{\partial^2}{\partial x^2}F(s,X_s)u_sD_s^\phi X_sds.
\end{align}
\end{theorem}

\section{Convergence}
\label{sec3}
In this section, we consider convergence of the solution of the
equation~\eqref{1.7}, the so-call {\it linear fractional
self-attracting diffusion}. The method used here is essentially due
to M. Cranston and Y. Le Jan~\cite{Crans}.
\begin{proposition}\label{th3.1}
The solution to the equation~\eqref{1.7} can be expressed as
\begin{equation}\label{seceq3-1}
X^H_t=X_0^H+\int_0^th(t,s)dB_s^H+\nu\int_0^th(t,s)ds,
\end{equation}
where
\begin{equation}\label{seceq3-2}
h(t,s)=
 \begin{cases}
 1-ase^{\frac12as^2}\int_s^te^{-\frac12au^2}du, &\text{$t\geq s$},\\
      0,&\text{$t<s$}
 \end{cases}
\end{equation}
for $s,t\geq 0$.
\end{proposition}
This proposition can also be obtained by the same method as Cranston
and Le Jan~\cite{Crans}. It follows from the Ito type formula that
\begin{align*}
F(X_t^H)&=F(0)+\int_0^tF^{'}(X_s^H)dX_s^H+\int_0^tF^{''}(X_s^H)
D^\phi_sX_s^Hds\\
&=F(0)+\int_0^tF^{'}(X_s^H)dX_s^H\\
&\qquad+2{H}(2H-1)\int_0^tF^{''}(X_s^H)
ds\int_0^sh(s,m)(s-m)^{2H-2}dm
\end{align*}
for $F\in C^2({\mathbb R})$ having polynomial growth. On the other
hand, an elementary calculation yields
\begin{equation}\label{seceq3-2-10}
h(s)=\lim_{t\uparrow\infty}h(t,s)=1-ase^{\frac{a}2s^2} \int_s^\infty
e^{-\frac{a}2u^2}du,
\end{equation}
which is continuous on $[0,\infty)$.
\begin{theorem}\label{th3.2}
The solution $X^H_t$ to~\eqref{1.7} converges in $L^2(\mu)$ to the
following element as $t\to \infty$:
$$
X^H_\infty\equiv \int_0^\infty h(s)dB_s^H+\nu\int_0^\infty h(s)ds.
$$
\end{theorem}
\begin{proof}
Clearly, we have
$$
|h(t,s_1)-h(s_1)||h(t,s_2)-h(s_2)|\leq
\frac1{t^2}s_1s_2e^{\frac{a}2(s_1^2+s_2^2)}e^{-at^2}
$$
for $s_1,s_2\leq t$ and $\left|\int_0^t[h(t,s)-h(s)]ds\right|\leq
\frac1{at}\to 0\quad (t\to \infty).$ From~\eqref{condition1} it
follows that
\begin{align}\label{eq3.4}
E&\left|\int_0^t[h(t,s)-h(s)]dB_s^H\right|^2\leq
\frac{2H}{at^{2-2H}}\to 0
\end{align}
as $t\to \infty$, which proves
\begin{align*}
E\left|X_t^H-X_\infty^H\right|^2\leq
2E\left|\int_0^t[h(t,s)-h(s)]dB_s^H\right|^2+
2\left|\int_0^t[h(t,s)-h(s)]ds\right|^2\to 0.
\end{align*}
This completes the proof.
\end{proof}

\begin{theorem}\label{th3.3}
The solution $X^H_t$ to~\eqref{1.7} converges to $X^H_\infty$ almost
surely as $t\to \infty$.
\end{theorem}
\begin{proof}
Without loss of generality, we may assume $\nu=0$. Note that by
Proposition~\ref{th3.1}
\begin{align*}
{X}_t^H-X_\infty^H&=\int_0^t[h(t,s)-h(s)]dB_s^H-\int_t^\infty
h(s)dB_s^H\\
&\equiv Y_t^H-\int_t^\infty h(s)dB_s^H,\quad t\geq 0.
\end{align*}
Thus, it is enough to show that $Y_t^H$ converges to $0$ almost
surely as $t\to \infty$.

For integer numbers $n$, $k$, $0\leq k<n$ we set
$Z^H_{n,k}=Y_{n+\frac{k}n}^H$. Then $Z^H_{n,k}$ is Gaussian, and we
have
\begin{align*}
E\left[(Z^H_{n,k})^2\right]=E\left[\left|\int_0^{n+\frac{k}n}
[h(t,s)-h(s)]dB_s^H\right|^2\right]\leq\frac{2H}{an^{2-2H}}
\end{align*}
by~\eqref{eq3.4}, and for any $\varepsilon>0$
$$
P(|Z^H_{n,k}|>\varepsilon)\leq
2\frac{e^{-\frac{1}{4H}\varepsilon^2an^{2-2H}}}
{\varepsilon{n}^{1-H}\sqrt{a\pi}}.
$$
Furthermore, for $s\in (0,1)$ we set
$R_s^{n,k}=Y_{n+\frac{k+s}n}-Y_{n+\frac{k}n}$. Then $R_s^{n,k},0\leq
s\leq 1$ is a Gaussian process and
$$
E\left[(R^{n,k}_s-R^{n,k}_{s^{'}})^2\right]\leq
\frac{|s-s^{'}|^{2H}}{n^{2H}}(1+a)=\frac{1+a}{n^{2H}}
E\left[(B^H_s-B^H_{s^{'}})^2\right]
$$
by~\eqref{eq3.4}. For any $\varepsilon>0$,
\begin{align*}
P\left(\sup_{0\leq s\leq 1}|R^{n,k}_s|>\varepsilon\right)&\leq
P\left(\frac{\sqrt{1+a}}{n^{H}}\sup_{0\leq s\leq
1}|B^H_s|>\varepsilon\right)\quad ({\text {by Slepian's lemma}})\\
&\leq (1+a)\frac{E\left[\sup_{0\leq s\leq
1}|B^H_s|^2\right]}{\varepsilon^2n^{2H}}\\
&\leq C_H(1+a)\frac{1}{\varepsilon^2n^{2H}}.
\end{align*}
Thus, the convergence with probability one follows from the
Borel-Cantelli Lemma and
$$
\left\{\sup_{n+\frac{k}n<t<n+\frac{k+1}n}|Y_t|>\varepsilon\right\}
\subseteq \{|Z^H_{n,k}|>\varepsilon/2\}\cup\left\{\sup_{0\leq s\leq
1}|R^{n,k}_s|>\varepsilon/2\right\}
$$
for all $k,n\geq 0$. This completes the proof of the theorem.
\end{proof}

\section{Local time and Meyer-Tanaka type formula}
\label{sec4}
In this section, we consider the linear fractional self-attracting
diffusion $X^H=\{X^H_t,0\leq t\leq T\}$ with $\nu=0$. It follows
that the process is a centered Gaussian process. We study the usual
local time and weighted local time of the process and obtain the
Meyer-Tanaka type formula of the weighted local time.

For $T\geq t\geq s\geq 0$ we put
$$
\sigma^2_t=E\left[({X_{t}^{H}})^2\right],\quad\sigma^2_{t,s}
=E\left[\left(X_{t}^{H}-X_{s}^{H}\right)^2\right].
$$
Then
$$
\sigma^2_t=\int_0^{t}\int_0^{t}h(t,u)h(t,v)\phi(u,v)dudv,\qquad
0\leq t\leq T
$$
and
\begin{equation}\label{sec4-eq4.100}
\sigma^2_{t,s}=\int_0^{t}\int_0^{t}\left[h(t,u)-h(s,u)\right]
\left[h(t,v)-h(s,v)\right]\phi(u,v).
\end{equation}
Noting that for all $t\geq s\geq 0$,
$$
\int_0^t\int_0^t\phi(u,v)dudv=t^{2H},\qquad
e^{-\frac{a}2(t^2-s^2)}\leq h(t,s)\leq 1,
$$
we get
\begin{equation}\label{sec4-eq4.1}
e^{-\frac{a}2t^2}t^{2H}\leq
\sigma_t^2=\int_0^t\int_0^th(t,u)h(t,v)\phi(u,v)dudv\leq t^{2H}.
\end{equation}

\begin{lemma}\label{lemma4.1}
For all $t\geq s\geq 0$ we have
\begin{equation}\label{sec4-eq4.2}
c_{a,H,T}(t-s)^{2H}\leq \sigma_{t,s}^2\leq C_{a,H,T}(t-s)^{2H},
\end{equation}
where $C_{a,H,T},c_{a,H,T}>0$ are two constants depending on
$a,H,T$.
\end{lemma}
\begin{proof}
For all $t\geq s\geq 0$ we have
\begin{align*}
\sigma_{t,s}^2&=\int_0^{t}\int_0^{t}\left[h(t,u)-h(s,u)\right]
\left[h(t,v)-h(s,v)\right]\phi(u,v)dudv\\
&=\int_s^{t}\int_s^{t}h(t,u)h(t,v)\phi(u,v)dudv+\\
&\qquad
\int_s^{t}\int_0^{s}h(t,u)\left[h(t,v)-h(s,v)\right]\phi(u,v)dudv+\\
&\qquad\int_0^{s}\int_s^{t}\left[h(t,u)-h(s,u)\right]
h(t,v)\phi(u,v)dudv+\\
&\qquad\int_0^{s}\int_0^{s}\left[h(t,u)-h(s,u)\right]
\left[h(t,v)-h(s,v)\right]\phi(u,v)dudv\\
&\equiv A_{[s,t]^2}+A_{[s,t]\times [0,s]}+A_{[0,s]\times
[s,t]}+A_{[0,s]^2}.
\end{align*}
On the other hand, for all $t\geq s\geq 0$ we have
\begin{align*}
A_{[s,t]\times [0,s]}&=A_{[0,s]\times
[s,t]}\\
&=-\left(\int_s^te^{-\frac{a}2w^2}dw\right)
\int_0^{s}aue^{\frac{a}2u^2}du\int_s^{t}h(t,v)\phi(u,v)dv,
\end{align*}
\begin{align*}
A_{[0,s]^2}=2\left(\int_s^te^{-\frac{a}2w^2}dw\right)^2
\int_0^{s}a^2ue^{\frac{a}2u^2}du\int_0^{u}ve^{\frac{a}2v^2}
\phi(u,v)dv.
\end{align*}
But, some elementary calculus can show that
\begin{align*}
C_Ha^2e^{-aT^2}s^{2+2H}(t-s)^{2}&\leq A_{[0,s]^2}\leq
2a^2s^{2+2H}(t-s)^{2},\\
e^{-\frac{a}2(t^2-s^2)}(t-s)^{2H}&\leq A_{[s,t]^2}\leq (t-s)^{2H},
\end{align*}
and
\begin{align*}
\lim_{s\uparrow t}\frac{A_{[s,t]\times
[0,s]}}{(t-s)^{2H}}=0,\qquad&\lim_{s\downarrow
0}\frac{A_{[s,t]\times [0,s]}}{(t-s)^{2H}}=0,
\end{align*}
which lead to
\begin{align*}
\lim_{s\uparrow t}\frac{\sigma_{t,s}^2}{(t-s)^{2H}}=1,\qquad
e^{-\frac{a}2t^2}t^{2H}\leq \lim_{s\downarrow
0}\frac{\sigma_{t,s}^2}{(t-s)^{2H}}\leq t^{2H}.
\end{align*}
It follows that there are two constants $c_{a,H,T},C_{a,H,T}>0$ such
that
\begin{align*}
c_{a,H,T}(t-s)^{2H}\leq \sigma_{t,s}^2\leq C_{a,H,T}(t-s)^{2H}.
\end{align*}
This completes the proof.
\end{proof}

From the lemma above, we see that
$$
\int_0^t\int_0^tE\left[(X_u^H-X_v^H)^2\right]^{-1/2}dudv<\infty
$$
holds for all $t\geq 0$, and furthermore, we can show that the
process $X^H=(X_t^H)_{0\leq t\leq T}$ is local nondeterminism for
every $0<T<\infty$, i.e. for $0\leq t_1<t_2<\cdots<t_n\leq T$,
\begin{equation}\label{sec4-eq4.3}
{\rm
Var}\left(\sum_{j=2}^nu_j(X_{t_j}^{H}-X_{t_{j-1}}^{H})\right)\geq
\kappa_0\sum_{j=2}^nu_j^2\sigma^2_{t_j,t_{j-1}}
\end{equation}
with a constant $\kappa_0>0$. Combining this with
Berman~\cite{Berman1,Berman2}, we obtain
\begin{proposition}
If $\nu=0$, then the solution $X^H$ of the equation~\eqref{1.7} has
continuous local time ${\mathscr L}_t^x$, $t\geq 0,x\in {\mathbb R}$
such that
$$
{\mathscr L}_t^x=\lim_{\varepsilon\downarrow
0}\frac1{2\varepsilon}\int_0^t1_{[x-\varepsilon,x+\varepsilon]}(X_s^H)ds
=\int_0^t\delta(X_s^H-x)ds,
$$
where $\delta(X^H_s-\cdot)$ denotes the delta function of $X_s^H$.
\end{proposition}
For $t\geq 0,x\in {\mathbb R}$ we now set
\begin{align*}
{\mathcal L}_t^x&=2{H}(2H-1)\int_0^t\delta(X_s^H-x)ds
\int_0^sh(s,m)(s-m)^{2H-2}dm.
\end{align*}
Then ${\mathcal L}_t^x$ is well-defined and
$$
{\mathcal L}_t^x=\int_0^t\delta(X_s^H-x)D_s^\phi X_s^Hds.
$$
The process $({\mathcal L}_t^x)_{t\geq 0}$ is called the {\it
weighted local time} of $X^H$ at $x\in {\mathbb R}$.

\begin{lemma}[Hu~\cite{Hu1}]\label{lemma3.2}
Let $Y$ be normally distributed with mean zero and variance
$\sigma^2\,(\sigma>0)$. Then the delta function $\delta(Y-\cdot)$ of
$Y$ exists uniquely and we have
\begin{equation}\label{sec3eq3.1}
\delta(Y-x)=\frac1{2\pi}\int_{\mathbb R}e^{i\xi(Y-x)}d\xi,\qquad
x\in {\mathbb R}.
\end{equation}
\end{lemma}

\begin{proposition}\label{prop3.2}
Assume that $t\in [0,T]$ is given. Then ${\mathcal L}_t^x$ and
${\mathscr L}_t^x$ are square integrable for all $x\in {\mathbb R}$
and we have
\begin{equation}\label{sec3eq3.4}
E\left[\left({\mathscr L}_t^x\right)^2\right]\leq C_{a,H,T}t^{2-2H},
\end{equation}
\begin{equation}\label{sec3eq3.5}
E\left[\left({\mathcal L}_t^x\right)^2\right]\leq C_{a,H,T}t^{2H}.
\end{equation}
\end{proposition}

\begin{proof}
We have by Lemma~\ref{lemma3.2}
\begin{align*}
E\left[\left({\mathscr L}_t^x\right)^2\right]&\leq
\frac{1}{(2\pi)^2}\int_0^t\int_0^tdudv\int_{{\mathbb
R}^2}E\left[e^{i(\xi X_u^H+\eta X_v^H)}\right]d\xi d\eta\\
&\leq \frac{2}{(2\pi)^2}\int_0^tdu\int_0^tdv\int_{{\mathbb R}^2}
e^{-\frac12{\rm Var}(\xi X_u^H+\eta X_v^H)}d\xi d\eta.
\end{align*}
Noting that
\begin{align*}
{\rm Var}(\xi X_u^H+\eta X_v^H)\geq
k\left[\xi^2\sigma_{u,v}^2+(\eta+\xi)^2\sigma_v^2\right]
\end{align*}
for a positive $k>0$ by local nondeterminacy of the process $X^H$,
we get
\begin{align*}
E\left[\left({\mathscr L}_t^x\right)^2\right]&\leq
\frac{2}{(2\pi)^2}\int_0^tdu\int_0^udv\int_{{\mathbb R}^2}
e^{-\frac{k}2(\xi^2\sigma_{u,v}^2+(\eta+\xi)^2\sigma_v^2)}
d\xi d\eta\\
&\leq \frac{1}{k\pi}\int_0^tdu\int_0^u
\frac{dv}{\sigma_{u,v}\sigma_v}\\
&\leq \frac1{k\pi
c_{a,T}}\int_0^tdu\int_0^u\frac{dv}{(u-v)^Hv^H}\leq
C_{a,H,T}T^{2-2H}.
\end{align*}
by Lemma~\ref{lemma4.1}. This obtains~\eqref{sec3eq3.4}. Similarly,
one can show that inequality~\eqref{sec3eq3.5} holds.
\end{proof}

\begin{theorem}\label{theorem3.1}
Let $X^H$ be the solution to the equation~\eqref{1.7} with Hurst
index $\frac12<H<1$, $X_0^H=z$, $\nu=0$ and let ${\mathcal L}$ be
the weighted local time of $X^H$. Suppose that $\Phi:{\mathbb
R}^+\to {\mathbb R}$ is a convex function having polynomial growth.
Then
\begin{equation}\label{sec3eq3.100}
\Phi(X_t^H)=\Phi(z)+\int_0^tD^{-}\Phi(X_s^H)dX_s^H+\int_{\mathbb
R}{\mathcal L}_t^x\mu_{\Phi}(dx),
\end{equation}
where $D^{-}\Phi$ denotes the left derivative of $\Phi$ and the
signed measure $\mu_\Phi$ is defined by
$$
\mu_{\Phi}([a,b])=D^{-}\Phi(b)-D^{-}\Phi(a),\qquad a<b,a,b\in
{\mathbb R}.
$$
\end{theorem}
\begin{proof}
For $\varepsilon>0$ and $x\in {\mathbb R}$ we set
$$
\Phi_\varepsilon(x)=\int_{\mathbb R}p_\varepsilon(x-y)\Phi(y)dy\quad
(\varepsilon>0),
$$
where $p_\varepsilon(x)=\frac1{\sqrt{2\pi\varepsilon}}
e^{-\frac1{2\varepsilon}x^2} $. Then $\Phi_\varepsilon\in C^2$ and
we have $ \lim_{\varepsilon\downarrow
0}\Phi_\varepsilon(x)=\Phi(x)$, $ \lim_{\varepsilon\downarrow
0}\Phi_\varepsilon^{'}(x)=D^{-}\Phi(x) $ for all $x\in {\mathbb R}$.
It follows that for all $\varepsilon>0$
\begin{align*}
\Phi_\varepsilon(X^H_t)&=\Phi_\varepsilon(z)+\int_0^t\Phi_\varepsilon^{'}
(X^H_s)dX_s^H+
2H(2H-1)\int_0^t\Phi_\varepsilon^{''}(X_s^H){\widetilde{h}}(s)ds.
\end{align*}

On the other hand, it is easy to see that $\Phi_\varepsilon(X_t^H)$
converges to $\Phi(X_t^H)$ almost surely, and
$\int_0^t\Phi_\varepsilon^{'}(X_s^H)X_s^Hds\to
\int_0^tD^{-}\Phi(X_s^H)X_s^Hds$ a.s., and furthermore,
$\int_0^t\Phi_\varepsilon^{'}(X_s^H)dB_s^H\to
\int_0^tD^{-}\Phi(X_s^H)dB_s^H$ in ${(\mathcal S)}^{*}$.

Finally, we have as $\varepsilon\to 0$
\begin{align*}
\int_0^t\Phi_\varepsilon^{''}(X_s^H){\widetilde{h}}(s)ds
&=\int_0^tds{\widetilde{h}}(s)\int_{\mathbb
R}\Phi_\varepsilon^{''}(x)\delta(X_s^H-x)dx\\
&\to \frac{1}{2H(2H-1)}\int_{\mathbb R}{\mathcal
L}_t^x\mu_{\Phi}(dx).
\end{align*}
This completes the proof.
\end{proof}

\begin{corollary}
Let $X^H$ be the solution to the equation~\eqref{1.7} with Hurst
index $\frac12<H<1$, $X_0^H=z$, $\nu=0$ and let ${\mathcal L}$ be
the weighted local time of $X^H$. Then the Tanaka formula
\begin{equation}\label{sec3eq3.200}
|X_t^H-x|=|X_0^H-x|+\int_0^t{\rm {sign}}(X_s^H-x)dX_s^H+{\mathcal
L}_t^x
\end{equation}
holds for all $x\in {\mathbb R}$.
\end{corollary}

\section{Self-intersection local time on ${\mathbb R}^2$}
\label{sec5}
In this section, we shall use the idea of Hu~\cite{Hu0} (see also
Hu-Nualart~\cite{Hu2}) to study the the renormalized
self-intersection local time of the linear fractional
self-attracting diffusion $X^H=(X^{H,1},X^{H,2})$ on ${\mathbb
R}^2$, where $X^{H,j}\;(j=1,2)$ is the solution of the equation
$$
X_t^{H,j}=B^{H,j}_t-a\int_0^t\int_0^u(X^{H,j}_u-X^{H,j}_v)dvdu
,\qquad 0\leq t\leq T
$$
with $a>0$ and two independent fractional Brownian motions
$B^{H,j}_{t},\,j=1,2$. Then we have
$$
X_t^{H,j}=\int_0^th(t,s)dB^{H,j}_t,\qquad j=1,2
$$
from Section~\ref{sec3}, and for all $s,t\geq 0$
$$
h(t,s)=
 \begin{cases}
 1-ase^{\frac12as^2}\int_s^te^{-\frac12au^2}du, &\text{$t\geq s$},\\
      0,&\text{$t<s$}.
 \end{cases}
$$

The renormalized self-intersection local time $\beta_T^H$ of the
process
$$
X^H_t=(X^{H,1}_t,X^{H,2}_t),\qquad 0\leq t\leq T
$$
is formally defined as
$$
\beta_T^H=\int_0^T\int_0^t\delta_0(X^H_t-X^H_s)dsdt-E\left[
\int_0^T\int_0^t\delta_0(X^H_t-X^H_s)dsdt\right],
$$
where $\delta_0$ is the delta function. For $\varepsilon>0$ we
define
$$
\beta_T^{H,\varepsilon}=
\int_0^T\int_0^tp_\varepsilon(X^H_t-X^H_s)dsdt,
$$
where
$$
p_\varepsilon(x)=\frac1{2\pi\varepsilon}
e^{-\frac{|x|^2}{2\varepsilon}},\quad x\in {\mathbb R}^2
$$
is the heat kernel. Then main object of this section is to explain
and prove Theorem~\ref{th5.1}.
\begin{theorem}\label{th5.1}
The random variable
$\beta_T^{H,\varepsilon}-E\left[\beta_T^{H,\varepsilon}\right]$
converges in $L^2$ as $\varepsilon$ tends to zero if
$\frac12<H<\frac3{4}$.
\end{theorem}

In order to prove the theorem we need some preliminaries. For $t\geq
s\geq 0,t^{'}\geq s^{'}\geq 0$ we now denote
$$
\sigma^2_{t,s}=E\left(X^{H,1}_t-X^{H,1}_s\right)^2,\quad
\mu=E(X^{H,1}_t-X^{H,1}_s)(X^{H,1}_{t^{'}}-X^{H,1}_{s^{'}})
$$
and
$$
d_H(s,t,s^{'},t^{'})=\sigma^2_{s,t}\sigma^2_{s^{'},t^{'}}- \mu^2.
$$
Then, by Lemma~\ref{lemma4.1} and Hu~\cite{Hu0} one can establish
the following lemma.
\begin{lemma}\label{lem5.1}

(1) For $0<s<s^{'}<t<t^{'}<T$, we have
\begin{equation}\label{sec5-eq3.4}
d_H(s,t,s^{'},t^{'})\geq
\kappa\left[(t-s)^{2H}(t^{'}-t)^{2H}+(t^{'}-s^{'})^{2H}
(s^{'}-s)^{2H}\right].
\end{equation}

(2) For $0<s^{'}<s<t<t^{'}<T$, we have
\begin{equation}\label{eq3.5}
d_H(s,t,s^{'},t^{'})\geq \kappa(t-s)^{2H}(t^{'}-s^{'})^{2H}.
\end{equation}

(3) For $0<s<t<s^{'}<t^{'}<T$, we have
\begin{equation}\label{eq3.6}
d_H(s,t,s^{'},t^{'})\geq \kappa(t-s)^{2H}(t^{'}-s^{'})^{2H},
\end{equation}
where $\kappa>0$ is an enough small constant.
\end{lemma}
\begin{lemma}\label{lem5.2}
For $0\leq x<y\leq T$ we set
\begin{align*}
h^{*}&(y,x,u,v)\\
&=\!\left[h(y,u)1_{(0,y]}(u)-h(x,u)1_{(0,x]}(u)\right]\! \left[
h(y,v)1_{(0,y]}(v)-h(x,v)1_{(0,x]}(v)\right],
\end{align*}
where $h(\cdot,\,\cdot)$ is defined in Section~\ref{sec3}. Then we
have
\begin{align*}
\int_0^{t^{'}}\int_0^{t^{'}}\bigl[h^{*}(t^{'},s,u,v)-
&h^{*}(t^{'},t,u,v)\bigr]\phi(u,v)dudv\\
&\leq C_{a,H,T}\left[(t^{'}-s)^{2H}-(t^{'}-t)^{2H}\right]
\end{align*}
for all $0\leq s\leq t\leq t^{'}\leq T$.
\end{lemma}
\begin{proof} For $0<u,v<T$ we have
\begin{align*}
h(t^{'},u)&1_{(0,t^{'}]}(u)-h(s,u)1_{(0,s]}(u)
=h(t^{'},u)1_{(s,t^{'}]}(u)+\left[h(t^{'},u)-h(s,u)
\right]1_{(0,s]}(u),
\end{align*}
and
\begin{align*}
h(t^{'},v)&1_{(0,t^{'}]}(v)-h(s,v)1_{(0,s]}(v)
=h(t^{'},v)1_{(s,t^{'}]}(v)+\left[h(t^{'},v)-h(s,v)
\right]1_{(0,s]}(v).
\end{align*}
So,
\begin{align*}
h^{*}(t^{'},s,u,v)&=h(t^{'},u)h(t^{'},v)1_{(s,t^{'}]^2}(u,v)\\
&+\left[h(t^{'},u)-h(s,u)\right]
\left[h(t^{'},v)-h(s,v)\right]1_{(0,s]^2}(u,v)\\
&-h(t^{'},u)\left[h(s,v)-h(t^{'},v)
\right]1_{(s,t^{'}]}(u)1_{(0,s]}(v)\\
&-h(t^{'},v)\left[h(s,u)-h(t^{'},u)
\right]1_{(0,s]}(u)1_{(s,t^{'}]}(v).
\end{align*}
Similarly, we also have
\begin{align*}
h^{*}(t^{'},t,u,v)&=h(t^{'},u)h(t^{'},v)1_{(t,t^{'}]^2}(u,v)\\
&+\left[h(t^{'},u)-h(t,u)\right]
\left[h(t^{'},v)-h(t,v)\right]1_{(0,t]^2}(u,v)\\
&-h(t^{'},u)\left[h(t,v)-h(t^{'},v)
\right]1_{(t,t^{'}]}(u)1_{(0,t]}(v)\\
&-h(t^{'},v)\left[h(t,u)-h(t^{'},u)
\right]1_{(0,t]}(u)1_{(t,t^{'}]}(v).
\end{align*}

On the other hand, for all $0<u,v\leq s\leq t\leq t^{'}\leq T$ we
set
\begin{align*}
\Delta(t^{'},t,s,u,v)\equiv[h(s,u)&-h(t^{'},u)][h(s,v)-h(t^{'},v)]
1_{(0,s]^2}(u,v)\\
&-[h(t,u)-h(t^{'},u)][h(t,v)-h(t^{'},v)]1_{(0,t]^2}(u,v)\\
&\hspace{-2cm}=a^2uve^{\frac{a}2(u^2-v^2)}\left\{\left(\int_s^{t^{'}}
e^{-\frac{a}2w^2}dw\right)^2-\left(\int_t^{t^{'}}
e^{-\frac{a}2w^2}dw\right)^2\right\}.
\end{align*}
Then for all $s<t\leq t^{'}\leq T$, we have
\begin{align*}
\lim_{s\downarrow 0}\frac{1}{(t^{'}-s)^{2H}-(t^{'}-t)^{2H}}
\int_0^{s}\int_0^{s}\Delta(t^{'},t,s,u,v)\phi(u,v)dudv=0
\end{align*}
and
\begin{align*}
\lim_{s\uparrow t}\frac{1}{(t^{'}-s)^{2H}-(t^{'}-t)^{2H}}
\int_0^{s}\int_0^{s}\Delta(t^{'},t,s,u,v)\phi(u,v)dudv=0,
\end{align*}
which implies that there is a constant $C_{a,H,T}>0$ such that
$$
\int_0^{s}\int_0^{s}\Delta(t^{'},t,s,u,v)\phi(u,v)dudv\leq
C_{a,H,T}\left[(t^{'}-s)^{2H}-(t^{'}-t)^{2H}\right].
$$
Combining these with
$$
0\leq h(t,u)-h(t^{'},u)\leq h(s,u)-h(t^{'},u)\leq 2,\qquad 0\leq
u\leq s\leq t,
$$
we get
\begin{align*}
\int_0^{t^{'}}&\int_0^{t^{'}}\bigl[h^{*}(t^{'},s,u,v)-
h^{*}(t^{'},t,u,v)\bigr]\phi(u,v)dudv\\
&\leq\int_0^{t^{'}}\int_0^{t^{'}}
h(t^{'},u)h(t^{'},v)\left(1_{(s,t^{'}]^2}(u,v)
-1_{(t,t^{'}]^2}(u,v)\right)\phi(u,v)dudv\\
&\qquad\quad +4\int_t^{t'}\int_s^t\phi(u,v)dudv\\
&\qquad\qquad+\int_0^{s}\int_0^{s}\Delta(t^{'},t,s,u,v)\phi(u,v)dudv\\
&\leq C_{a,H,T}\left[(t^{'}-s)^{2H}-(t^{'}-t)^{2H}\right].
\end{align*}
This completes the proof.
\end{proof}

The proof similar to Lemma~\ref{lem5.2}, by decomposing the function
$$
h^{*}(t^{'},s,u,v)-
h^{*}(t^{'},t,u,v)+h^{*}(s^{'},t,u,v)-h^{*}(s^{'},s,u,v),
$$
one can show that the following lemma holds for all $0\leq s\leq
t\leq s^{'}\leq t^{'}\leq T$.
\begin{lemma}\label{lem5.3}
Under the assumptions of Lemma~\ref{lem5.2}, for all $0\leq s\leq
t\leq s^{'}\leq t^{'}\leq T$ we have
\begin{align*}
\int_0^{t^{'}}\int_0^{t^{'}}&\bigl[h^{*}(t^{'},s,u,v)-
h^{*}(t^{'},t,u,v)+h^{*}(s^{'},t,u,v)-h^{*}(s^{'},s,u,v)
\bigr]\phi(u,v)dudv\\
&\leq C_{a,H,T}\left[(t^{'}-s)^{2H}-(t^{'}-t)^{2H}+
(s^{'}-t)^{2H}-(s^{'}-s)^{2H}\right].
\end{align*}
\end{lemma}

\begin{lemma}\label{lem5.4}
Let $\sigma^2_{t,s}$ and $\mu$ be as in the proof of
Theorem~\ref{th5.1}. Then we have
\begin{equation*}
\int_{\mathbb T}\frac{\mu^2dsdtds^{'}dt^{'}}{
d_H(s,t,s^{'},t^{'})(\sigma^2_{t,s}\sigma^2_{t^{'},s^{'}})}<\infty
\end{equation*}
if $\frac12<H<\frac3{4}$.
\end{lemma}

Lemma~\ref{lem5.4} is a consequence of the lemmas above and Lemma 15
in Hu~\cite{Hu2} (see also Hu~\cite[pp.245--247]{Hu0}). In fact,
Lemma~\ref{lem5.2} and Lemma~\ref{lem5.3} imply that the estimate
$$
\mu\leq C_{a,H,T}\left[|t-s'|^{2H}-|t'-t|^{2H}+|t'-s|^{2H}-
|s-s'|^{2H}\right]
$$
holds for all $(s,s^{'},t,t^{'})\in {\mathbb T}$. Thus,
Lemma~\ref{lem5.4} follows from Lemma~\ref{lem5.1} and Lemma 15 in
Hu~\cite{Hu2} (see also Hu~\cite[pp.245--247]{Hu0}).

Now we can prove Theorem~\ref{th5.1}.

\begin{proof}[Proof of Theorem~\ref{th5.1}]
Clearly, as $\varepsilon$ tends to zero
$\beta_T^{H,\varepsilon}-E\left[\beta_T^{H,\varepsilon}\right]$
converges in $L^2$ if and only if
\begin{equation}\label{sec5-eq5.4}
Var(\beta_T^{H,\varepsilon})=
E\bigl[(\beta_T^{H,\varepsilon})^2\bigr]-
\left(E(\beta_T^{H,\varepsilon})\right)^2
\end{equation}
tends to a constant. Now let us show that
$Var(\beta_T^{H,\varepsilon})$ converges as $\varepsilon$ tends to
zero. Using the classical equality
$$
p_\varepsilon(x)=\frac1{(2\pi)^{2}}\int_{{\mathbb
R}^2}e^{i<\xi,x>}e^{-\varepsilon\frac{|\xi|^2}{2}}d\xi,
$$
one can obtain
\begin{equation}\label{eq3.1}
\beta_T^{H,\varepsilon}=\frac1{(2\pi)^{2}}\int_0^T\int_0^t
\int_{{\mathbb
R}^2}e^{i<\xi,X^{H}_t-X^H_s>}e^{-\varepsilon\frac{|\xi|^2}{2}} d\xi
dsdt.
\end{equation}
Combining this with the facts $<\xi,X^H_t-X^H_s>\sim
N(0,|\xi|^2\sigma^2_{t,s})$ and
$$
E\left[e^{i<\xi,X^H_t-X^H_s>}\right]=
e^{-\frac12|\xi|^2\sigma^2_{t,s}},
$$
$$
\int_{{\mathbb
R}^2}e^{-\frac12|\xi|^2\left(\varepsilon+\sigma^2(t,s)\right)}d\xi=
\frac{2\pi}{\varepsilon+\sigma^2_{t,s}},
$$
we get
\begin{align}\tag*{}
E\left[\beta_T^{H,\varepsilon}\right]&=\int_0^T\int_0^t
E\left(p_\varepsilon(X^H_t-X^H_s)\right)dsdt\\
\label{eq3.2} &=\frac1{2\pi}\int_0^T\int_0^t\left(\varepsilon+
\sigma^2_{t,s}\right)^{-1}dsdt.
\end{align}
Denote ${\mathbb
T}=\{(s,t,s^{'},t^{'}):\;0<s<t<T,\,0<s^{'}<t^{'}<T\}$, then
according to the representation~\eqref{eq3.1} we get
\begin{align*}
E\left[(\beta_T^{H,\varepsilon})^2\right]&\!\!=
\frac1{(2\pi)^{4}}\!\int_{\mathbb T}\!\int_{{\mathbb
R}^{4}}\!\!Ee^{i<\xi,X^{H}_t-X^H_s>+i<\eta,X^H_{t'}- X^H_{s'}>}
e^{-\varepsilon\frac{|\xi|^2+|\eta|^2}{2}}d\xi d\eta dsdtds'dt'.
\end{align*}
Noting that
$$
<\!\xi,X^H_t-X^H_s\!>+<\!\eta,X^H_{t'}-X^H_{s'}\!>\sim
N\!\left(0,|\xi|^2\sigma^2_{t,s}\!\!+
|\eta|^2\sigma^2_{t',s'}\!+2\mu<\xi,\eta>\right)
$$
for any $\xi,\eta\in {\mathbb R}^2$, we can write
\begin{align*}
E\left[(\beta_T^{H,\varepsilon})^2\right]&=\frac1{(2\pi)^{4}}
\int_{\mathbb T}\int_{{\mathbb
R}^{4}}e^{-\frac12\left((\sigma^2_{t,s}+\varepsilon)|\xi|^2+
2\mu<\xi,\eta>+(\sigma^2_{t',s'}+\varepsilon) |\eta|^2
\right)}d\xi d\eta dsdtds'dt'\\
&=\frac1{(2\pi)^2}\int_{\mathbb
T}\left((\sigma^2_{t,s}+\varepsilon)(\sigma^2_{t',s'}+
\varepsilon)-\mu^2\right)^{-d/2}dsdtds'dt'
\end{align*}
for all $\varepsilon>0$. It follows from~\eqref{eq3.2} that
\begin{align*}
E\left[(\beta_T^{H,\varepsilon})^2\right]-
\left(E\beta_T^{H,\varepsilon}
\right)^2&=\frac1{(2\pi)^{2}}\int_{\mathbb
T}\left[\left((\sigma^2_{t,s}+\varepsilon)(\sigma^2_{t',s'}+
\varepsilon)-\mu^2\right)^{-1}\right.-
\\
&\qquad\qquad\left.\left((\varepsilon+ \sigma^2_{t,s})(\varepsilon+
\sigma^2_{t',s'})\right)^{-1}\right]dsdtds'dt'\\
&\hspace{-2cm}=\frac1{(2\pi)^{2}}\int_{\mathbb
T}\frac{\mu^2dsdtds'dt'}{\left((\sigma^2_{t,s}+\varepsilon)
(\sigma^2_{t',s'}+ \varepsilon)-\mu^2\right)(\varepsilon+
\sigma^2_{t,s})(\varepsilon+ \sigma^2_{t',s'})}.
\end{align*}
Thus, the theorem follows from Lemma~\ref{lem5.4}.
\end{proof}

\section*{Acknowledgement}
The authors would like to thank the anonymous referees for the
careful reading of the manuscript and many helpful comments.

\end{document}